\documentclass[10pt,letterpaper,reqno]{amsart}
\usepackage{amsmath,amssymb,mathtools,booktabs,enumitem}
\usepackage[expansion=false]{microtype}
\usepackage{cite}
\usepackage[hidelinks]{hyperref}
\theoremstyle{plain}
\newtheorem{theorem}{Theorem}[section]
\newtheorem{proposition}[theorem]{Proposition}
\newtheorem{corollary}[theorem]{Corollary}
\newtheorem{lemma}[theorem]{Lemma}
\theoremstyle{definition}
\newtheorem{definition}[theorem]{Definition}
\newtheorem{example}[theorem]{Example}
\newtheorem{remark}[theorem]{Remark}
\newcommand{\Pow}{\mathcal P}
\newcommand{\Nrel}{\mathbf N}
\newcommand{\Ftwo}{\mathbb F_2}
\newcommand{\aff}{\operatorname{aff}}
\hypersetup{
  pdftitle={Finite bases for full power semirings of finite nilpotent semigroups},
  pdfauthor={Lili Wang, Qingrui Yin, Aifa Wang}}

\title[Full power semirings of finite nilpotent semigroups]
{Finite bases for full power semirings of finite nilpotent semigroups}
\author[L. Wang]{Lili Wang}
\address{School of Mathematical Sciences, Chongqing University of Technology,
Chongqing 400054, P.R. China}
\email{wllaf@cqut.edu.cn}
\author[Q. Yin]{Qingrui Yin}
\address{School of Mathematical Sciences, Chongqing University of Technology,
Chongqing 400054, P.R. China}
\email{13996559059@163.com}
\author[A. Wang]{Aifa Wang}
\address{School of Mathematical Sciences, Chongqing University of Technology,
Chongqing 400054, P.R. China}
\email{wangaf@cqut.edu.cn}
\date{}

\begin{document}
\begin{abstract}
We investigate finite equational bases for full power semirings of
finite semigroups, including the empty set, in the constant-free signature
with addition and multiplication. For a nontrivial finite nilpotent
semigroup, we associate a finite relational structure recording the ordered
products that are nonzero. We prove that the full power semiring is finitely
based if and only if this structure has finite duality, and relate this
condition to first-order definability and dismantling of the square of its
core. Quantitative bounds connect obstruction size with the number of
variables required in an identity basis. In the commutative case of
nilpotency index $d$, the criterion reduces to the existence of an element
with nonzero $(d-1)$st power. We also establish a nonfinite-basis obstruction
for semigroups with a two-element group ideal and prove a finite lifting
theorem. These results yield a direct-product criterion and a five-element
counterexample to sufficiency of the identity-fibre condition. The arguments
use equational logic, finite relational duality, and explicit algebraic
constructions.
\end{abstract}
\maketitle

\section{Introduction and main results}

An additively idempotent semiring (ai-semiring) is an algebra
$(A,+,\cdot)$ whose additive reduct is a commutative idempotent semigroup,
whose multiplicative reduct is a semigroup, and whose multiplication
distributes over addition on both sides.  If $S$ is a finite semigroup, its
\emph{full power semiring} is
\[
 \Pow(S)=\bigl(\{A:A\subseteq S\},+,\cdot\bigr),\qquad
 A+B=A\cup B,\quad AB=\{ab:a\in A,\ b\in B\}.
\]
The empty set is included throughout.  We use the constant-free binary
signature $(+,\cdot)$; neither the empty set nor a zero of $S$ is named by a
constant.  A finite basis means a finite set of valid identities from which
all valid identities follow by equational logic.  Equivalently, one may work
relative to the usual finite set of ai-semiring identities. We use the
standard notions of identity, variety, and equational consequence from
universal algebra; see Burris and Sankappanavar \cite{BS}.

The finite basis problem for power semirings is part of the study of
equational axiomatizability for ai-semirings. Dolinka \cite{D09} developed
sufficient conditions for inherent nonfinite basability, while Jackson,
Ren, and Zhao \cite{JRZ} investigated nonfinitely based ai-semirings whose
semigroup reducts are finitely based. These works also contain earlier
formulations of power-semiring finite-basis questions. The general problem
of characterizing finite semigroups with finitely based full power
semirings is posed again by Yue, Ren, and Gao in their preprint
\emph{Explicit equational bases for the power semirings of $S_7$}
(\href{https://arxiv.org/abs/2609.19957v1}{arXiv:2609.19957v1},
2026, Problem~1.1).

It is essential to distinguish the full power semiring from the nonempty
power semiring. For finite solvable non-Dedekind groups, Gusev and Volkov
proved nonfinite basability of the full power semiring \cite{GV}. The
group classification of Gao, Ren, Shao, and Yue in
\emph{The finite basis problem for the power semirings of finite groups}
(\href{https://arxiv.org/abs/2608.15793v1}{arXiv:2608.15793v1}, 2026)
concerns the nonempty version. Shao and Gao obtained the full-power
classification for finite groups, namely finite basis exactly in orders at
most two, in their preprint
\emph{Nonfinitely based intervals of power semiring varieties}
(\href{https://arxiv.org/abs/2609.26019v1}{arXiv:2609.26019v1},
2026, Corollary~1.3).

There is related work on nil-semigroups.  Wu and Zhao proved sufficient
nonfinite-basis conditions for nonempty power semirings of finite
nil-semigroups defined by commutative words
\cite{WZ}.  The presence of the empty set changes the reduction to local
constraints, so the full version cannot simply be read off from the nonempty
one. We give a direct argument for the full power semiring.

The relational language used here is the homomorphism formulation of
constraint satisfaction, developed systematically by Feder and Vardi
\cite{FV}. Homomorphism dualities for finite relational structures were
studied by Ne\v{s}et\v{r}il and Tardif \cite{NT}. Their connection with
first-order definability is due to Atserias \cite{Ats}, and the finite-tree
and dismantling characterizations that we use are presented by Larose,
Loten, and Tardif \cite{LLT}. In our setting the relations record precisely
which ordered short products are nonzero. The key step is to connect
homomorphism obstructions for these relations with absorption identities
under arbitrary subset assignments. Contextual closure then produces
finite quotient algebras witnessing failure of every proposed bound on
the number of variables in an identity basis.

We obtain a classification for all finite nilpotent semigroups, a simpler
criterion in the commutative case, and quantitative variable bounds. We
then extend the obstruction method to semigroups with a two-element group
ideal. A lifting theorem gives finite bases in a complementary situation
and yields the direct-product result below.

A semigroup with zero $0$ is \emph{nilpotent} if $S^d=\{0\}$ for some
positive integer $d$, where $S^d$ is the set of all products of $d$ elements.
Let $S\ne\{0\}$ be a finite nilpotent semigroup, and let $d\ge2$ be least
such that $S^d=\{0\}$.  Put $r=d-1$ and define the finite relational
structure
\begin{equation}\label{eq:N}
 \Nrel(S)=\bigl(S\setminus\{0\};R_1,\ldots,R_r\bigr),\qquad
 R_j=\{(a_1,\ldots,a_j):a_1\cdots a_j\ne0\}.
\end{equation}
The order of the factors is retained, repetitions are allowed, and
$R_1=S\setminus\{0\}$.

A finite relational structure $D$ has \emph{finite duality} if there are
finitely many finite structures $O_1,\ldots,O_t$ of the same signature such
that, for every finite structure $H$,
\[
 H\nrightarrow D\quad\Longleftrightarrow\quad
 (\exists i)\ O_i\to H.
\]
Here arrows denote homomorphisms preserving all relations; the obstruction
family is allowed to be empty.

\begin{theorem}[Nilpotent classification]\label{thm:nil}
Let $S\ne\{0\}$ be a finite nilpotent semigroup.  Then $\Pow(S)$ is finitely
based if and only if $\Nrel(S)$ has finite duality.  If finite duality fails,
then for every $k\ge1$ there is a finite ai-semiring satisfying all
identities of $\Pow(S)$ involving at most $k$ variables but failing another
identity of $\Pow(S)$.
\end{theorem}

Finite duality for a finite relational structure is decidable by the
core-square dismantling criterion of Larose, Loten, and Tardif
\cite[Theorems~2.5 and 5.7]{LLT}.  The substantive point here is the
equivalence between that relational condition and the identity theory of the
full power semiring.

\begin{corollary}[Commutative case]\label{cor:comm}
If $S\ne\{0\}$ is finite and commutative with nilpotency index $d$, then
\[
 \Pow(S)\ \text{is finitely based}
 \quad\Longleftrightarrow\quad
 \exists a\in S\ \ (a^{d-1}\ne0).
\]
\end{corollary}

The theorem also gives quantitative information.  Let $b(D)$ be the least
$b\ge0$ such that every finite structure not mapping to $D$ has an induced
substructure with at most $b$ vertices that still does not map to $D$; put
$b(D)=0$ if every finite structure maps to $D$, and $b(D)=\infty$ if no such
finite bound exists.  Let $v(A)$ be the least possible maximum number of
variables in a finite relative identity basis for $A$, with $v(A)=\infty$
when no finite basis exists.  The fixed ai-semiring identities are not counted
in $v(A)$; the countermodels below satisfy them.  In the finitely based case,
we obtain

\begin{equation}\label{eq:bounds-intro}
 b(\Nrel(S))\le(d-1)\bigl(v(\Pow(S))+2\bigr),\qquad
 v(\Pow(S))\le\max\{d+1,b(\Nrel(S))\}.
\end{equation}

\begin{theorem}[Two-element group direct products]\label{thm:product}
For every finite commutative nilpotent semigroup $N$,
\[
 \Pow(N\times C_2)\ \text{is finitely based}
 \quad\Longleftrightarrow\quad
 \Pow(N)\ \text{is finitely based}.
\]
For arbitrary finite nilpotent $N$, finite basis of $\Pow(N\times C_2)$
implies finite basis of $\Pow(N)$.  If $N\ne\{0\}$ has nilpotency index $d$
and $a^{d-1}\ne0$ for some $a\in N$, then $\Pow(N\times C_2)$ is finitely
based without a commutativity assumption.
\end{theorem}

The final part of the paper proves a phase-sensitive nonfinite-basis
obstruction and applies it to a five-element semigroup.  These examples also
show that a semigroup whose full power semiring is finitely based can have
a subsemigroup and a Rees quotient whose full power semirings are nonfinitely
based.

\section{Absorption identities and contextual closure}

Let $X^+$ be the set of nonempty multiplicative words, let $|w|$ denote word
length, and let $c(w)$ denote the variable support of $w$.  By distributivity,
every term is a finite nonempty sum of words.  We write
\[
 w\preceq Q\qquad\Longleftrightarrow\qquad w+Q\approx Q,
\]
We identify a finite nonempty set of words with its idempotent sum and write
$c(Q)=\bigcup_{q\in Q}c(q)$.  Every identity is equivalent
to finitely many absorption identities, obtained by absorbing each word on
one side into the other side.

The following lemma is the full-power instance of Lemma~2.2 in the
Shao--Gao preprint \href{https://arxiv.org/abs/2609.26019v1}{arXiv:2609.26019v1}
mentioned in the introduction. We include the direct proof.

\begin{lemma}[Empty-set elimination]\label{lem:empty}
Suppose $\Pow(S)\models w\preceq Q$, and put $W=c(w)$.  Delete from $Q$ all
words using variables outside $W$, and call the resulting set $Q_W$.  Then
$Q_W\ne\varnothing$ and
\[
 \Pow(S)\models w\preceq Q_W.
\]
\end{lemma}
\begin{proof}
Assign the variables outside $W$ the empty set.  This proves validity after
the deletion.  If no word remained, assign every variable in $W$ an arbitrary
nonempty singleton.  The left side is then nonempty while the right side is
empty, a contradiction.  The empty right side in this argument is only a
value of a term, not a constant in the language.
\end{proof}

Fix $A=\Pow(S)$, a finite auxiliary semigroup $U$, and $k\ge1$.  For every
valid absorption identity $w\preceq Q$ of $A$ involving at most $k$ variables,
every assignment of subsets $A_i\subseteq U$, and every pair of contexts
$u,v\in U^1$, impose the rule
\begin{equation}\label{eq:closure-rule}
 uQ(A_i)v\subseteq Z\quad\Longrightarrow\quad uw(A_i)v\subseteq Z .
\end{equation}
The formal identity in $U^1$ merely means that a context is absent on that
side.  Let $c_k(Z)$ be the least superset of $Z$ closed under all these
rules.  Since $U$ is finite, it is obtained by iterating the rules until
stability.

\begin{lemma}[The closure quotient]\label{lem:closure}
Write $c=c_k$.  Then
\begin{gather}
 uc(X)v\subseteq c(uXv),\qquad c(X)c(Y)\subseteq c(XY),\label{eq:nucleus}\\
 c(c(X)\cup c(Y))=c(X\cup Y),\qquad
 c(c(X)c(Y))=c(XY).\label{eq:quotient}
\end{gather}
The closed subsets of $U$ form a finite ai-semiring $B_k$ with operations
\[
 C\oplus D=c(C\cup D),\qquad C\otimes D=c(CD),
\]
and $X\mapsto c(X)$ is a surjective semiring homomorphism.  This quotient
satisfies every identity of $A$ involving at most $k$ variables.
\end{lemma}
\begin{proof}
If a closure derivation is given, adding fixed left and right contexts to
each step gives another instance of \eqref{eq:closure-rule}.  Induction on
the length of the derivation yields the first inclusion in
\eqref{eq:nucleus}.  Hence $c(X)Y\subseteq c(XY)$.  For $x\in c(X)$,
\[
 x c(Y)\subseteq c(xY)\subseteq c(c(X)Y)\subseteq c(XY),
\]
which gives the second inclusion.  The equalities in \eqref{eq:quotient}
follow from monotonicity, extensivity, and idempotence together with
\eqref{eq:nucleus}.  The quotient operations are therefore well defined,
and all ai-semiring identities pass through the surjective homomorphism.

For a valid absorption identity $w\preceq Q$ and closed-set values $C_i$,
the set $c(Q(C_i))$ contains the original premise $Q(C_i)$.  The rule with
empty contexts therefore puts $w(C_i)$ inside this closure; taking closure
once more gives the absorption identity in the quotient.  Decomposing an
identity into absorption identities proves the last assertion.
\end{proof}

\begin{remark}\label{rem:map}
If $h:U\to S$ is a semigroup homomorphism and $T\subseteq S$, then
$h^{-1}(T)$ is closed under all the rules above.  Indeed, $h$ induces a
homomorphism of full power semirings, and both the target absorption identity
and its left and right multiplicative contexts are valid.  No subsemigroup
assumption on $T$ is needed.
\end{remark}

\section{The nilpotent criterion}\label{sec:nil}

\subsection{Long-word normal forms and the finite short library}

Fix $S\ne\{0\}$ and its nilpotency index $d$, and write $r=d-1$.  A word of
length at least $d$ will be called \emph{long}; a word of length less than $d$
will be called \emph{short}.  A long word evaluates to $\{0\}$ if every
variable in its support has nonempty value, and to the empty set otherwise.
Consider the finite set
$L_d$ consisting of
\begin{align}
 x_1\cdots x_i x_{i+1}\cdots x_d
 &\approx x_1\cdots x_{i+1}x_i\cdots x_d
 &&(1\le i<d),\label{eq:Lswap}\\
 x_1\cdots x_d&\approx x_1^2x_2\cdots x_d,\label{eq:Ldup}\\
 x_1\cdots x_dy&\preceq x_1\cdots x_d.\label{eq:Ldrop}
\end{align}
These identities are valid in $\Pow(S)$, including assignments involving the
empty set.

\begin{lemma}\label{lem:normal}
The identities $L_d$ imply:
\begin{enumerate}[label=(\roman*),nosep]
\item any two long words with the same support are equal;
\item if $u$ and $v$ are long and $c(v)\subseteq c(u)$, then $u\preceq v$.
\end{enumerate}
\end{lemma}
\begin{proof}
Any adjacent pair in a long word occurs in some window of length $d$, so
\eqref{eq:Lswap} permits arbitrary permutations.  Move a chosen letter to the
first position of such a window and apply \eqref{eq:Ldup}; this adds one
occurrence.  For two words with the same support, raise each multiplicity to
the coordinatewise common upper bound and then permute, proving (i).

For (ii), add occurrences and permute to derive $u\approx vt$, where $t$
may be empty.  If $t$ is nonempty, split $v$ into $d$ nonempty blocks and
substitute $t$ for $y$ in \eqref{eq:Ldrop}; this gives $vt\preceq v$.
\end{proof}

On $r$ variables let $\mathcal W_r$ contain all words of length less than
$d$ and one long representative for every nonempty support.  Put
\begin{equation}\label{eq:Psi}
 \Psi_S=\left\{\,u\preceq\sum_{q\in Q}q:
       u\in\mathcal W_r,\ \varnothing\ne Q\subseteq\mathcal W_r,\
       \Pow(S)\models u\preceq\sum_{q\in Q}q\,\right\}.
\end{equation}
This is a finite set determined by the multiplication table of $S$.
Lemma~\ref{lem:empty} and Lemma~\ref{lem:normal} show that every valid
absorption identity with a short left word follows from $L_d\cup\Psi_S$.

\subsection{Long left words and finite duality}

For a nonempty set $Q$ of short words, let $H_Q$ be the relational structure
on a vertex set $W\supseteq c(Q)$ obtained by putting
$(i_1,\ldots,i_j)$ in $R_j(H_Q)$ for every word
$x_{i_1}\cdots x_{i_j}\in Q$ of length $j$.

\begin{lemma}\label{lem:csp}
Let $w$ be long, let $W=c(w)$, and suppose that every word in $Q$ is short
and uses only variables in $W$.  Then
\[
 \Pow(S)\models w\preceq Q
 \quad\Longleftrightarrow\quad
 H_Q\nrightarrow\Nrel(S).
\]
\end{lemma}
\begin{proof}
If a homomorphism $H_Q\to\Nrel(S)$ exists, assign its nonzero singleton
values to the variables.  Every right-hand word is nonzero, whereas the
left-hand word is zero, so absorption fails.
 
Conversely, assume that there is no homomorphism but that an assignment of
subsets violates the absorption identity.  The left side must be nonempty
and equal to $\{0\}$, so all left variables are nonempty and no right word
contains zero in its value.  Every variable occurring on the right therefore
has a value not containing $0$.  Choose one element from each such value
appearing on the right, and assign arbitrary nonzero elements to isolated
variables.  The chosen elements define a homomorphism
$H_Q\to\Nrel(S)$, a contradiction.
\end{proof}

\begin{proof}[Proof of Theorem~\ref{thm:nil}, sufficiency]
Let $O_1,\ldots,O_t$ be a finite obstruction family for $\Nrel(S)$.  Remove
isolated vertices from the obstructions; this does not change completeness
because $\Nrel(S)$ is nonempty.  Every remaining obstruction contains a
relation tuple.  Let $Q_{O_i}$ be the sum of its tuple words, and let $u_i$
be any long word whose support is exactly $V(O_i)$.  Lemma~\ref{lem:csp}
gives the valid identities
\begin{equation}\label{eq:obstacle-identity}
 u_i\preceq Q_{O_i}.
\end{equation}
We claim that $L_d\cup\Psi_S$ together with these finitely many identities
is a basis.

Let $w\preceq Q$ be valid.  First apply Lemma~\ref{lem:empty} to restrict
the right-hand support.  A short left word is handled by $\Psi_S$.  If a
long right-hand word $q$ occurs, Lemma~\ref{lem:normal}(ii) gives
$w\preceq q\preceq Q$.  We are left with a long left word and only short
right-hand words.  By Lemma~\ref{lem:csp}, $H_Q\nrightarrow\Nrel(S)$, so
some $O_i$ maps to $H_Q$.  Substitute the corresponding variables in
\eqref{eq:obstacle-identity}.  Its right side consists of words from $Q$,
while its left side is long and has support contained in $c(w)$.  Another
application of Lemma~\ref{lem:normal}(ii) absorbs $w$ into that left word
and then into $Q$.
\end{proof}

\subsection{Local short-word constructions}

\begin{lemma}[Bounded local obstructions]\label{lem:bounded}
A nonempty finite relational structure $D$ has finite duality if and only if
$b(D)<\infty$.  If finite duality fails, then for every $M$ there is a finite
structure $H\nrightarrow D$ all of whose induced substructures on at most
$M$ vertices map to $D$.
\end{lemma}
\begin{proof}
If each member of the obstruction family has at most $M$ vertices, map an obstruction into
$H$ and take the induced substructure on its image.  It has at most $M$
vertices and cannot map to $D$, since otherwise the obstruction would map to
$D$.  Conversely, for a fixed finite signature there are only finitely many
isomorphism types on at most $M$ vertices.  Taking all types that do not map
to $D$ gives a finite obstruction family.
\end{proof}

Fix $k\ge1$.  Suppose that
$H\nrightarrow\Nrel(S)$ but every induced substructure on at most
$M=r(k+2)$ vertices maps to $\Nrel(S)$.  Write $V=V(H)$; then
$|V|>M\ge d$.  Define
\[
 U_r(V)=\{\text{ordered words on $V$ of lengths }1,\ldots,r\}
 \ \dot\cup\ \{\infty\},
\]
where concatenation of short words of total length greater than $r$ is
the absorbing element $\infty$.  Let $F$ be the set of tuple words of $H$,
and form the contextual closure $c_k(F)$.

\begin{lemma}[Invariant before the first overflow]\label{lem:localnil}
Before $\infty$ first appears in a closure derivation from $F$, every
short word $f$ already present has the following property: for every
$C\supseteq c(f)$ and every homomorphism
$\theta:H[C]\to\Nrel(S)$, evaluation of the letters of $f$ under $\theta$
is nonzero in $S$.
\end{lemma}
\begin{proof}
The map $\theta$ extends to a semigroup homomorphism
\[
 h_\theta:U_r(C)\to S,\qquad h_\theta(\infty)=0,
\]
because $S^{r+1}=\{0\}$.  Tuple words in $F$ have nonzero images.
Suppose a closure step adds a short word $f$.  Choose a factor selection
in a realization of $uw(A)v$ that produces $f$.  For each variable retain
the set of all short factors selected at its occurrences, and assign the
variables outside the left support the empty set.  All remaining letters
lie in $c(f)\subseteq C$.  The new premise is a subset of the old premise,
contains no $\infty$, and consists of words whose supports lie in $C$.
The induction hypothesis says that every such premise word has nonzero
$h_\theta$-image.  Validity of the target absorption identity, including
the chosen contexts, then forces $h_\theta(f)\ne0$.
\end{proof}

\begin{lemma}\label{lem:nooverflow}
Under the preceding hypotheses, $\infty\notin c_k(F)$.
\end{lemma}
\begin{proof}
Assume that a first closure step adds $\infty$, and let $Z$ be the set
before that step.  The left variables have nonempty values.  Set variables
outside the left support to the empty set.  By Lemma~\ref{lem:empty}, the
right side still has a word; hence an actual context equal to $\infty$ would
already put $\infty$ in the premise, which is impossible.
 
For each left variable choose one factor: retain $\infty$ if it occurs, and
otherwise choose a short factor of maximum length.  The resulting
singleton choices still make the left side equal to $\infty$: an existing
absorbing factor is retained, and otherwise the total length does not
decrease.  Every premise output is still a short word in $Z$.  The union
$C$ of the selected finite letters and the context letters has size at most
$r(k+2)$.  Choose a homomorphism $\theta:H[C]\to\Nrel(S)$.  By
Lemma~\ref{lem:localnil}, all premise outputs have nonzero image, whereas
the left side, which is $\infty$, maps to zero.  This contradicts validity
of the absorption identity with its contexts.
\end{proof}

\begin{proof}[Proof of Theorem~\ref{thm:nil}, necessity]
If finite duality fails, Lemma~\ref{lem:bounded} supplies such an $H$ for
every $k$.  Let $Q_H$ be the sum of the tuple words and let $w_H$ be a word
containing every vertex exactly once.  Since $|V|>M\ge d$, Lemma~\ref{lem:csp}
gives the valid identity $w_H\preceq Q_H$.
 
In the closure quotient $B_k$ of Lemma~\ref{lem:closure}, assign
$x_v$ the closed singleton $c_k(\{v\})$.  The left side of this identity is
$c_k(\{\infty\})$, while the right side is $c_k(F)$.  Lemma~\ref{lem:nooverflow}
shows that the identity fails in $B_k$; Lemma~\ref{lem:closure} shows that
$B_k$ satisfies all target identities involving at most $k$ variables.
Thus a finite basis, with $k$ at least its maximum number of variables,
cannot exist.
\end{proof}

For the one-point zero semigroup, $\Pow(\{0\})$ has two elements.  Relative
to the ai-semiring identities, one may take
\[
 xy\approx yx,\qquad x^2\approx x,\qquad xy\preceq x
\]
as a basis; completeness follows directly from support inclusion.

\section{Variable bounds, relational obstructions, and commutativity}
\label{sec:bounds}

\subsection{Quantitative finite-basis constraints}

\begin{theorem}[Obstruction size and variables]\label{thm:quant}
Let $S\ne\{0\}$ be finite nilpotent of index $d$, and put $r=d-1$.
If the identity theory of $\Pow(S)$ is generated, possibly by an infinite
set, by identities involving at most $k$ variables, then
\[
 b(\Nrel(S))\le r(k+2).
\]
Conversely, if $b=b(\Nrel(S))<\infty$, then $\Pow(S)$ has a finite relative
basis using at most $\max\{d+1,b\}$ variables.
\end{theorem}
\begin{proof}
If the first bound failed, there would be an $H\nrightarrow\Nrel(S)$ all of
whose induced substructures on at most $r(k+2)$ vertices map to $\Nrel(S)$.
The closure construction in the proof of Theorem~\ref{thm:nil} uses only
this one $H$ and produces an ai-semiring satisfying all valid identities on
at most $k$ variables but failing another valid identity, contradicting the
generation hypothesis.
 
Conversely, take all non-mapping structures on at most $b$ vertices as a
finite obstruction family.  The long-word identities $L_d$ use at most
$d+1$ variables, $\Psi_S$ uses at most $r$, and each obstruction identity
uses at most $b$ variables.  Repetitions of existing variables make the left
word long, so no new variables are needed.
\end{proof}

In particular, if every induced substructure of an $H$ on at most
$r(k+2)$ vertices maps to $\Nrel(S)$ while $H$ itself does not, then the
corresponding long-left/short-right absorption identity cannot follow from
the target identities on at most $k$ variables.  It rules out every equational derivation from the valid identities
on at most $k$ variables.

\subsection{A noncommutative sufficient condition}

\begin{proposition}\label{prop:maxpower}
Let $S\ne\{0\}$ be finite nilpotent of index $d$.  If some
$a^{d-1}\ne0$, then $\Pow(S)$ is finitely based, without a commutativity
assumption.  In fact $L_d\cup\Psi_S$ is a finite relative basis.
\end{proposition}
\begin{proof}
The condition $a^{d-1}\ne0$ implies $a^j\ne0$ for every $1\le j<d$.
Consider a valid absorption identity with a long left word.  After
Lemma~\ref{lem:empty}, assign every left variable the singleton $\{a\}$.
The left side is zero and every short right word is nonzero, so the right
side must contain a long word.  Lemma~\ref{lem:normal}(ii) then proves the
identity.  Short left words are handled by $\Psi_S$.
\end{proof}

\subsection{The tree-obstruction argument in the commutative case}

We use one standard consequence of finite duality: a finite relational
structure with finite duality admits a complete obstruction family consisting
of finitely many \emph{relational trees}. This is part of the
homomorphism-duality theory of Ne\v{s}et\v{r}il and Tardif \cite{NT};
we use the formulation in \cite[Theorem~2.5]{LLT}.  The incidence bipartite multigraph of a relational
tree is a tree; each relation tuple is a tuple vertex, and every coordinate
occurrence gives an edge.  Thus a tuple in a tree has no repeated coordinate,
and each later tuple meets the part already constructed in exactly one
vertex.

\begin{proof}[Proof of Corollary~\ref{cor:comm}]
Sufficiency follows from Proposition~\ref{prop:maxpower}.  For necessity,
put $r=d-1$ and suppose that every $s^r=0$ but $S^r\ne\{0\}$.
The assumptions imply $r\ge2$.  Choose
$(a_1,\ldots,a_r)\in R_r$.  Let $L$ be the one-vertex structure with
$R_r(L)=\{(v,\ldots,v)\}$ and all other relations empty.  Then
$L\nrightarrow\Nrel(S)$.
 
If $\Nrel(S)$ had finite duality, the tree obstruction theorem would give a
relational tree $O$ with $O\to L$ but $O\nrightarrow\Nrel(S)$.  Since all
relations other than $R_r$ are empty in $L$, $O$ can contain only
$R_r$-tuples.  Every such tree nevertheless maps to $\Nrel(S)$.  Assign
the vertices of a root tuple successively to $a_1,\ldots,a_r$.  When a new
tuple is attached, exactly one vertex has already been assigned, say the
value $a_i$.  Put $a_i$ in the shared coordinate and place the remaining
$r-1$ factors, in any order, on the new vertices.  Commutativity makes the
product of every tuple equal to the nonzero product $a_1\cdots a_r$.
A tree with no tuple maps to any nonzero element.  This contradiction shows that
$\Nrel(S)$ has no finite duality, and Theorem~\ref{thm:nil} gives
nonfinite basis.
\end{proof}

The argument does not require the $a_i$ to be distinct: one only removes one
copy of the shared colour from the multiset.  It also does not require a
grading by word length; commutativity is used only to preserve nonzero
products under the required permutations.

\begin{corollary}\label{cor:permuted}
Let $S$ be finite nilpotent of index $d=r+1$, suppose that $s^r=0$ for
every $s\in S$, and assume that some $a_1,\ldots,a_r$ have nonzero product
in every permutation.  Then $\Pow(S)$ is nonfinitely based.
\end{corollary}
\begin{proof}
The tree extension above needs only nonzeroness of every permutation, not
equality of the resulting products.
\end{proof}

\subsection{Two boundary examples}

\begin{example}[The commutativity hypothesis is genuine]\label{ex:directed}
Let $L=\{a,b,c,0\}$, with the only nonzero binary product $ab=c$.
All products of length three are zero, so associativity holds and every
element has square zero.  The only nontrivial relation of $\Nrel(L)$ is the
directed edge $a\to b$; $c$ is isolated.  A finite directed graph maps to
this structure exactly when it contains no directed walk of length two.
Necessity is immediate.  For sufficiency, assign $a$ to vertices with an
outgoing edge, $b$ to vertices with an incoming edge, and any value to the
remaining vertices.  The absence of a length-two directed walk prevents a
vertex from receiving both prescriptions.  A directed path of length two is
therefore a finite obstruction, and Theorem~\ref{thm:nil} gives a finite
basis for $\Pow(L)$.
\end{example}

\begin{example}[Nonfinite basis is not inherited in this way]\label{ex:nonhereditary}
Let $T=\{a,b,c,0\}$ with $ab=ba=c$ as the only nonzero products.  It is
commutative of index three and every square is zero, so $\Pow(T)$ is
nonfinitely based.  Now let
\[
 S=\{a,b,c,u,v,0\},\qquad ab=ba=c,\quad u^2=v,
\]
with all other binary products zero.  All products of length three are zero,
so $S$ is associative, while $u^2\ne0$.  Proposition~\ref{prop:maxpower}
therefore gives a finite basis for $\Pow(S)$.  Nevertheless $T$ is a
subsemigroup of $S$, and $\{u,v,0\}$ is an ideal whose Rees quotient is
isomorphic to $T$.  The inclusion and quotient map induce, respectively, an
embedding and a surjective homomorphism of full power semirings.  Thus the
finite basis property is not preserved by these subsemirings or homomorphic
images.
\end{example}

\subsection{First-order definability and dismantling}

We record the logical and structural forms of the criterion. For a finite
relational structure $D$, write $\operatorname{CSP}(D)$ for the class of
finite structures $H$ of the same signature with $H\to D$. This class is
\emph{first-order definable} if a single first-order sentence in the
relational signature defines membership among finite structures.

A \emph{retract} $C$ of $D$ is an induced substructure admitting a
homomorphism $D\to C$ that fixes $C$ pointwise. A \emph{core} is a retract
with no proper retract. Finite cores exist and are unique up to isomorphism.
In a relational structure, an element $x$ is \emph{dominated} by $y\ne x$
if replacing $x$ by $y$ at any one coordinate of any relation tuple gives
another tuple of that relation. A structure \emph{dismantles} to an induced
substructure if the latter is obtained by successive deletions of dominated
elements. The relations in $C^2$ are defined coordinatewise, and its
diagonal is $\Delta_C=\{(c,c):c\in C\}$.

\begin{corollary}[Equivalent relational criteria]\label{cor:logical}
Let $S\ne\{0\}$ be finite nilpotent of index $d$, put $D=\Nrel(S)$,
and let $C$ be a core of $D$. The following conditions are equivalent:
\begin{enumerate}[label=(\roman*),nosep]
\item $\Pow(S)$ is finitely based;
\item $D$ has finite duality;
\item $\operatorname{CSP}(D)$ is first-order definable;
\item $D$ has a finite complete obstruction family consisting of relational
trees;
\item $C^2$ dismantles to $\Delta_C$.
\end{enumerate}
\end{corollary}
\begin{proof}
Theorem~\ref{thm:nil} gives (i)$\Leftrightarrow$(ii). The equivalence of
(ii), (iii), and (iv) is \cite[Theorem~2.5]{LLT}; the first-order
characterization originates in the work of Atserias \cite{Ats}.
The equivalence with (v) follows from \cite[Theorem~5.7]{LLT} and the
core case established immediately before that theorem. Indeed, $D$ and
$C$ admit homomorphisms to each other, so they define the same
homomorphism class of finite structures.

For completeness, the sentence in (iii) can be written directly from
any obstruction family $O_1,\ldots,O_t$ in (ii). Delete isolated vertices
from these obstructions. For a remaining obstruction $O$, put
\begin{equation}\label{eq:canonical-query}
 \chi_O=
 \exists(x_v)_{v\in V(O)}\,
 \bigwedge_{j=1}^{d-1}\ \bigwedge_{(v_1,\ldots,v_j)\in R_j(O)}
 R_j(x_{v_1},\ldots,x_{v_j}).
\end{equation}
No distinctness conditions are imposed on the variables. Consequently
$H\models\chi_O$ if and only if $O\to H$. The sentence
$\bigwedge_{i=1}^t\neg\chi_{O_i}$ therefore defines
$\operatorname{CSP}(D)$; for an empty obstruction family it is the
identically true sentence.
\end{proof}

The equivalences in (ii)--(v) are established relational results; the
connection with the power-semiring condition (i) is supplied by
Theorem~\ref{thm:nil}. Since $D$ is finite, its retracts and the possible
deletions in $C^2$ form finite collections. Thus the criterion also implies
decidability of finite basability for full power semirings of finite
nilpotent semigroups.

\section{A phase obstruction for a two-element group ideal}

Suppose that a finite semigroup $S$ contains an ideal
$G=\{e,g\}\cong C_2$ and that $S^D\subseteq G$ for some $D\ge2$.  For
$s\in S$, both $se$ and $es$ lie in $G$, and
\[
 se=e(se)=(es)e=es.
\]
Consequently
\[
 \rho(s)=se,\qquad
 \rho(st)=\rho(s)\rho(t),\qquad
 sh=\rho(s)h,\quad hs=h\rho(s)\quad(h\in G),
\]
so $\rho:S\to G$ is a retraction.  We identify $e$ and $g$ with phases
$0$ and $1$ in $\Ftwo$.  Every product of length at least $D$ lies in $G$
and equals the product of the corresponding phase projections.

\begin{definition}[A prescribed-phase instance]
Fix $h\in G$.  A prescribed-phase instance $H$ consists of a finite vertex
set $V$, a phase map $\nu:V\to\Ftwo$, and finitely many ordered tuples of
lengths $1,\ldots,D-1$.  A satisfying assignment
$\theta:V\to S$ must obey $\rho(\theta(v))=\nu(v)$ and must make the product
of every listed tuple different from $h$.  Repeated coordinates are allowed.
The induced subinstance on a vertex set retains the prescribed phases and all
tuples contained in that set.  The empty induced instance is declared
satisfiable.
\end{definition}

\begin{theorem}[Prescribed-phase obstruction]\label{thm:phase}
Suppose that for some $h\in G$ and every $M\ge1$ there is an unsatisfiable
prescribed-phase instance $H$ whose induced subinstances on at most $M$
vertices are all satisfiable.  Then $\Pow(S)$ is nonfinitely based and its
identity theory has no basis with a bounded number of variables.
More precisely, for $k\ge1$, if
one unsatisfiable instance has every induced subinstance on at most
$2(D-1)(k+2)$ vertices satisfiable, then the identities of $\Pow(S)$ on at
most $k$ variables do not generate its full identity theory.
\end{theorem}

\subsection{Valid guard identities}

Put $r=D-1$, write $V_i=\nu^{-1}(i)$, and let $Q_H$ be the sum of the tuple
words.  An unsatisfiable instance has at least one tuple: both fibres
$\rho^{-1}(e)$ and $\rho^{-1}(g)$ are nonempty, so phase prescriptions alone
are always satisfiable.  Choose an odd integer $t\ge D$ and a new variable $z$.  If the
forbidden element is $e$, use
\begin{equation}\label{eq:guard0}
 z^{2t}\prod_{v\in V}x_v^{2t}
 \preceq z^t+\sum_{v\in V_0}z^tx_v+\sum_{v\in V_1}x_v^t+Q_H .
\end{equation}
If the forbidden element is $g$, use
\begin{equation}\label{eq:guard1}
 z^t\prod_{v\in V}x_v^{2t}
 \preceq z^{2t}+\sum_{v\in V_0}x_v^t+\sum_{v\in V_1}z^tx_v+Q_H .
\end{equation}
An empty indexed sum means that the corresponding terms are omitted; the
right side remains nonempty.  Fix any order for the product over $V$.

\begin{lemma}\label{lem:guards}
If $H$ is unsatisfiable, the corresponding identity
\eqref{eq:guard0} or \eqref{eq:guard1} is valid in $\Pow(S)$.
\end{lemma}
\begin{proof}
If one variable is empty, the left side is empty.  Suppose all values are
nonempty.  If the projection of $z$ is the whole group, its displayed power
term covers the left side.  If the projection of some $x_v$ is the whole
group, the displayed long term covering that variable does so.  Since the
left side is long and lies in $G$, it remains to consider singleton
projections.
 
For \eqref{eq:guard0}, the left side is $e$.  If $z$ has phase zero, $z^t$
supplies $e$.  If $z$ has phase one and a vertex has a phase different from
its prescribed phase, the term $z^tx_v$ or $x_v^t$ supplies $e$.  In the
remaining case all vertex phases are prescribed correctly.  Selecting one
element from each variable value gives a forbidden tuple by unsatisfiability,
so $e\in Q_H$.
 
For \eqref{eq:guard1}, if $z$ has phase zero then both the left side and
$z^{2t}$ are $e$.  If $z$ has phase one, the left side is $g$, and every
phase error is covered by the displayed long term for that vertex.  With all
phases correct, unsatisfiability again gives $g\in Q_H$.
\end{proof}

\subsection{The auxiliary semigroup and exclusion of the forbidden group element}

Let $H$ be as in Theorem~\ref{thm:phase}.  Form an auxiliary semigroup $U$
consisting of all ordered short words of lengths $1,\ldots,r$ together with
a disjoint copy $\{\epsilon,\gamma\}$ of $C_2$.  A letter has phase $\nu(v)$,
and a short word has the sum of its letter phases.  Concatenation remains a
short word when its length is at most $r$ and otherwise becomes the group
element with the same phase.  Products involving a group factor are
multiplied by adding phases.  For three short factors, both bracketings give
their concatenation if its total length is at most $r$, and otherwise the
group element of the total phase.  If there is a group factor, both
bracketings are evaluated in $C_2$.  Thus multiplication is associative.

Let $F$ be the set of tuple words, let $\eta$ be the group element
corresponding to the forbidden value $h$, and let $\bar\eta$ be the other
group element.  Put $F^+=F\cup\{\bar\eta\}$ and take its contextual closure
under all target absorption identities involving at most $k$ variables.

If $\theta$ satisfies $H[C]$, it induces a homomorphism
\begin{equation}\label{eq:phase-map}
 h_\theta:U(C)\to S,\qquad h_\theta(\epsilon)=e,\quad h_\theta(\gamma)=g,
\end{equation}
by evaluating short words in order.  Pure short products that overflow have
length at least $D$ and hence lie in $G$ with the correct phase; products
containing group elements follow from the retraction identities above.

Before $\eta$ enters the closure, every short word $f$ already present has
the invariant that $h_\theta(f)\ne h$ for every $C\supseteq c(f)$ and every
satisfying $\theta:H[C]\to S$.  The initial tuple words have this property.
For a closure step producing $f$, fix a factor selection realizing it and,
for each variable, retain all short factors selected at its occurrences,
assigning the other variables the empty set.  The reduced premises remain in
the preceding closure stage.  Their short outputs have support in $c(f)$
and are not mapped to $h$ by induction; a group output can only be
$\bar\eta$.  Evaluating the valid target identity through
\eqref{eq:phase-map} therefore proves the invariant.

\begin{lemma}\label{lem:phase-exclusion}
The forbidden group element $\eta$ does not belong to $c_k(F^+)$.
\end{lemma}
\begin{proof}
Suppose a first closure step adds $\eta$.  The left and right contexts can
be represented by two additional variables, so the step comes from a valid
absorption identity on at most $\ell=k+2$ variables.  Set variables outside
the left support to the empty set.  For each nonempty left-variable value,
retain the longest short factor of phase zero, the longest short factor of
phase one (when present), and every group element originally in that value.

Every factor selection producing $\eta$ can be replaced by representatives of
the same phase of no smaller length, while group factors are kept.  Thus the
left side still produces $\eta$.  The union $C$ of the retained letters has
at most $2r(k+2)$ vertices.  Choose a satisfying assignment on $H[C]$.
The invariant maps every reduced short premise away from $h$, and the only
group premise is $\bar\eta$.  The left side maps to $h$, contradicting
validity of the target absorption identity.
\end{proof}

\begin{proof}[Proof of Theorem~\ref{thm:phase}]
In the closure quotient, assign $z$ the closed singleton $c_k(\{\gamma\})$
and $x_v$ the closed singleton $c_k(\{v\})$.  The left side of the
appropriate guard identity is $c_k(\{\eta\})$.  Every displayed long
right-hand term evaluates to $\bar\eta$, so the right side is
$c_k(F^+)$.  Lemma~\ref{lem:phase-exclusion} says that this valid identity
fails in the quotient, while Lemma~\ref{lem:closure} says that all target
identities on at most $k$ variables hold there.
\end{proof}

\begin{corollary}[A necessary condition on the identity fibre]\label{cor:fibre}
Let $T=\rho^{-1}(e)$.  If $\Pow(S)$ is finitely based, then $\Pow(T)$ is
finitely based.
\end{corollary}
\begin{proof}
The element $e$ is a zero of $T$, and $T^D=\{e\}$.  If $T=\{e\}$ there is
nothing to prove.  Otherwise let $d\le D$ be the nilpotency index of $T$.
If $\Pow(T)$ were nonfinitely based, Theorem~\ref{thm:nil} and
Lemma~\ref{lem:bounded} would give arbitrarily large locally satisfiable
instances for $\Nrel(T)$.  Give all vertices phase zero, keep the
constraints of lengths below $d$, omit the others, and forbid $e$.
Any vertex occurring in a constraint must take a value different from $e$.
Values at isolated vertices may be replaced by any element of
$T\setminus\{e\}$.  Hence satisfiability is equivalent to the existence of
a homomorphism to $\Nrel(T)$.  These are instances of
Theorem~\ref{thm:phase}, contradicting finite basis of
$\Pow(S)$.
\end{proof}

\section{Two-phase lifting and finite-basis constructions}

This section proves the sufficient direction in Theorem~\ref{thm:product}
and records explicit finite axioms.

\subsection{The absorption criterion for \texorpdfstring{$C_2$}{C2}}

\begin{lemma}\label{lem:affine}
Let $w$ be a nonempty word with support $W$, let
$Q=q_1+\cdots+q_s$ be nonempty with $c(q_j)\subseteq W$, and let
$\alpha,\beta_j\in\Ftwo^W$ be the parity vectors of $w,q_j$, respectively.
Then
\[
 \Pow(C_2)\models w\preceq Q
 \quad\Longleftrightarrow\quad
 \bigcup_j c(q_j)=W,\qquad
 \alpha\in\aff\{\beta_1,\ldots,\beta_s\}.
\]
\end{lemma}
\begin{proof}
If a left variable is absent on the right, assign it the whole group and
all other variables the identity singleton; absorption fails.  For singleton
assignments $b\in\Ftwo^W$, validity says that some
$(\beta_j+\alpha)\cdot b=0$.  If $\alpha$ is outside the affine hull, put
$a=\beta_1+\alpha$ and
$L=\operatorname{span}\{\beta_j+\beta_1\}$.  A linear functional vanishing
on $L$ and taking value one on $a$ supplies a $b$ for which all these
dot products are one, a contradiction.
 
Conversely, an empty left-variable value makes the left side empty.
Assume that all values are nonempty.  If a variable value is the whole group, support coverage makes
some right word evaluate to the whole group.  Otherwise all values are
singletons.  The affine condition is equivalent to
$\alpha=\sum_{j\in J}\beta_j$ for a nonempty set $J$ of odd cardinality.
Over $\Ftwo$, an odd sum of bits equals one of the summands, so the left
value equals a right value.
\end{proof}

The identities in the following corollary belong to the class of
semilattice-ordered semigroups studied by Ren and Zhao \cite{RZ}.
We include a direct completeness proof because the absorption argument is
also needed for the lifting theorem.

\begin{corollary}\label{cor:c2basis}
Relative to the ai-semiring identities, $\Pow(C_2)$ has the finite basis
\[
 xy\approx yx,\qquad x^3\approx x.
\]
\end{corollary}
\begin{proof}
The four group subsets satisfy these identities.  Commutativity and
$x^3\approx x$ reduce each positive multiplicity to one or two according
to its parity, so words with the same support and parity are provably equal.
For a valid absorption identity, first
delete outside variables and apply Lemma~\ref{lem:affine} to obtain an odd
index set $J$.  Put
\[
 v=\left(\prod_{j\in J}q_j\right)\left(\prod_{j=1}^s q_j^2\right).
\]
The words $v$ and $w$ have the same support and parity, so the two displayed
identities prove $v\approx w$.  Moreover
$v\preceq Q^{|J|+2s}\approx Q$ by distributivity and reduction of odd powers.
\end{proof}

\subsection{A finite lifting criterion}

\begin{theorem}[Short-word lifting]\label{thm:lift}
Let $S$ have an ideal $G\cong C_2$ with $S^d\subseteq G$, $d\ge2$.  Suppose
there are elements $a_0,a_1\in S$ of phases zero and one such that
\begin{equation}\label{eq:lift}
 \{a_0,a_1\}^{\,j}\cap G=\varnothing\qquad(1\le j<d).
\end{equation}
Then $\Pow(S)$ is finitely based.
\end{theorem}
\begin{proof}
Let $E_d$ consist of all adjacent transposition identities in a length-$d$
window and
\begin{equation}\label{eq:Ed}
 x_1\cdots x_d\approx x_1^3x_2\cdots x_d .
\end{equation}
Long-word values depend only on group projections, and
$\Pow(C_2)\models x^3\approx x$, so $E_d$ is valid.  The window argument
from Lemma~\ref{lem:normal} allows arbitrary permutations and adds two
occurrences of any chosen variable.  Hence two long words with the same
support and parity vector can be made equal by $E_d$.
 
Add
\begin{equation}\label{eq:triple}
 U_dV_dW_d\preceq U_d+V_d+W_d,
\end{equation}
where $U_d,V_d,W_d$ are length-$d$ words on pairwise disjoint variable
sets.  This is valid because the product of any three elements of $C_2$ is
one of them; if a set value is empty, the left side is empty.  Splitting
each of three long words into $d$ nonempty blocks gives the corresponding
instance of \eqref{eq:triple}.  Consequently, if $R$ is any nonempty sum
of long words, distributivity yields $R^3\preceq R$.  Multiplicative
monotonicity and induction give
\begin{equation}\label{eq:oddR}
 R^{2m+1}\preceq R\qquad(m\ge0).
\end{equation}
 
On $r=d-1$ variables take a finite word library containing all short words
and one long representative for each nonempty support and parity vector.
Let $\mathcal W'_r$ denote this library and set
\[
 \Psi'_S=\left\{u\preceq\sum_{q\in Q}q:
 u\in\mathcal W'_r,\ \varnothing\ne Q\subseteq\mathcal W'_r,\
 \Pow(S)\models u\preceq\sum_{q\in Q}q\right\}.
\]
If a valid absorption identity has a short left word, empty-set elimination
restricts its right side to at most $r$ variables.  Long-word normalization
then reduces it to a member of $\Psi'_S$ after renaming variables.
 
Now let $w$ be long.  Restrict the right support and split the right side
into its long part $Q_L$ and short part.  For arbitrary group-subset values
$B_i$, lift them to $A_i=\{a_b:b\in B_i\}$.  Condition \eqref{eq:lift}
keeps every short right word outside $G$, while long words are evaluated by
their phase projections.  Nonempty singleton phase assignments show that
$Q_L$ is nonempty, and the same lifting for arbitrary phase subsets gives
$\Pow(C_2)\models w\preceq Q_L$.  Write $Q_L=q_1+\cdots+q_s$.
By Lemma~\ref{lem:affine}, choose an odd
index set $J$ and form
\[
 v=\left(\prod_{j\in J}q_j\right)\left(\prod_{j=1}^s q_j^2\right).
\]
The words $v$ and $w$ have the same support and parity, so $E_d\vdash
w\approx v$, while \eqref{eq:oddR} gives
$v\preceq Q_L^{|J|+2s}\preceq Q_L$.  Thus
\[
 E_d\cup\{\eqref{eq:triple}\}\cup\Psi'_S
\]
is a finite relative basis.
\end{proof}

\begin{proof}[Proof of Theorem~\ref{thm:product}]
The ideal of $N\times C_2$ is $\{0\}\times C_2$, and its identity fibre is
$N\times\{e\}\cong N$.  The necessary implication follows from
Corollary~\ref{cor:fibre}, without commutativity.  If $a^{d-1}\ne0$, take
$a_0=(a,e)$ and $a_1=(a,g)$.  Every product of fewer than $d$ such elements
has first coordinate $a^j\ne0$, so \eqref{eq:lift} holds and
Theorem~\ref{thm:lift} gives finite basis.  For commutative $N$,
Corollary~\ref{cor:comm} says that this is exactly the condition for
$\Pow(N)$ to be finitely based.  If $N$ is the one-point zero semigroup,
the product is $C_2$ and Corollary~\ref{cor:c2basis} applies.
\end{proof}

\begin{remark}
The coordinate projection
\[
 \Pow(N\times C_2)\longrightarrow\Pow(N)\times\Pow(C_2)
\]
loses the pairing information inside a subset and is generally not an
isomorphism.  The proof above works with arbitrary subset assignments and
does not make this identification.
\end{remark}

\section{A five-element example and further questions}

\begin{proposition}\label{prop:five}
Let $S_5=\{e,g,a,b,c\}$, with
$\rho(e)=\rho(c)=e$ and $\rho(g)=\rho(a)=\rho(b)=g$.  Apart from
$ab=ba=c$, let every binary product be the product of the two group phases.
Then $S_5$ is commutative, $S_5^3=C_2$,
$\Pow(\rho^{-1}(e))$ is finitely based, and $\Pow(S_5)$ is nonfinitely based.
\end{proposition}
\begin{proof}
The multiplication is displayed in Table~\ref{tab:five}.

\begin{table}[htbp]
\caption{Multiplication in $S_5$.}\label{tab:five}
\centering
\setlength{\tabcolsep}{12pt}
\begin{tabular}{cccccc}
\toprule
$\cdot$&$e$&$g$&$a$&$b$&$c$\\
\midrule
$e$&$e$&$g$&$g$&$g$&$e$\\
$g$&$g$&$e$&$e$&$e$&$g$\\
$a$&$g$&$e$&$e$&$c$&$g$\\
$b$&$g$&$e$&$c$&$e$&$g$\\
$c$&$e$&$g$&$g$&$g$&$e$\\
\bottomrule
\end{tabular}
\end{table}

Every binary product is either in the group or is $c$; multiplying once
more always enters the group and gives the product of the three phases.
Thus all bracketings of a triple agree, and the table is associative.
The group is contained in $S_5^3$, so $S_5^3=C_2$.
 
The identity fibre $\{e,c\}$ has all binary products equal to $e$, so its
full power semiring is finitely based by Proposition~\ref{prop:maxpower}.
In the odd fibre $\{g,a,b\}$, a binary product is different from $e$ exactly
for the ordered pairs $(a,b)$ and $(b,a)$.
 
For every $M$, take an odd cycle of length greater than $M$, prescribe phase
one at all vertices, and forbid $e$ on every oriented edge.  A satisfying
assignment would have to alternate $a$ and $b$ around the cycle, which is
impossible.  Every induced subgraph on at most $M$ vertices is a disjoint
union of paths and isolated vertices and can be alternately assigned $a,b$.
Theorem~\ref{thm:phase} therefore gives nonfinite basis.
\end{proof}

For this example one also obtains a direct family of valid identities.  For
every odd $n\ge3$, with subscripts read cyclically,
\begin{equation}\label{eq:odd-cycle}
 \prod_{i=1}^n x_i^6
 \preceq \sum_{i=1}^n x_i^3+\sum_{i=1}^n x_ix_{i+1}.
\end{equation}
If a variable is empty, the left side is empty.  Mixed phase projections are
covered by the cubic terms; an even-phase variable has cube $e$; and if all
variables have odd phase, the absence of $e$ from the edge sum would produce
a two-colouring of the odd cycle.  Hence \eqref{eq:odd-cycle} is valid.

A converse to Theorem~\ref{thm:phase} would require control of the
linear relations between phases and their interaction with short
products. It remains open whether bounded local obstructions for all
prescribed-phase instances suffice for finite basability under
$S^D\subseteq C_2$. For arbitrary finite semigroups, several idempotents
and regular components introduce further interactions. The present
finite-duality classification applies to nilpotent semigroups, while
Theorem~\ref{thm:product} treats only the direct products specified there.

For a noncommutative nilpotent $N$, we have proved that finite basis of
$\Pow(N\times C_2)$ implies finite basis of $\Pow(N)$, and the converse under
the maximal pure-power condition.  Example~\ref{ex:directed} shows that
finite basis of $\Pow(N)$ alone does not imply that extra condition.  Closing
this gap, and obtaining a classification for arbitrary finite semigroups,
remain open.

\end{document}